\documentclass[10pt,a4paper]{amsart} 
\usepackage{graphicx,latexsym,amssymb,color}
\theoremstyle{plain} 
\newtheorem{theorem}{Theorem}[section] 
\newtheorem{lemma}[theorem]{Lemma} 
\newtheorem{proposition}[theorem]{Proposition} 
 
\newtheorem*{question}{Question} 
\theoremstyle{definition}
 
\theoremstyle{remark} 
 
\newtheorem{claim}[theorem]{Claim}

\numberwithin{equation}{section}
\numberwithin{figure}{section}
\newcommand{\bd}{\begin{description}}   
\newcommand{\ed}{\end{description}} 
\newcommand{\ba}{\begin{array}}      \newcommand{\ea}{\end{array}} 
\newcommand{\bc}{\begin{center}}     \newcommand{\ec}{\end{center}} 
\newcommand{\be}{\begin{enumerate}}  \newcommand{\ee}{\end{enumerate}} 
\newcommand{\beq}{\begin{eqnarray}}  \newcommand{\eeq}{\end{eqnarray}} 
\newcommand{\beQ}{\begin{eqnarray*}} \newcommand{\eeQ}{\end{eqnarray*}} 
\newcommand{\bi}{\begin{itemize}}    \newcommand{\ei}{\end{itemize}}

\begin{document} 
 \title[Local moves for links with common sublinks 
]{Local moves for links with common sublinks} 

\author[J.B. Meilhan]{Jean-Baptiste Meilhan} 
\address{Institut Fourier, Universit\'e Grenoble 1 \\
         100 rue des Maths - BP 74\\
         38402 St Martin d'H\`eres , France}
	 \email{jean-baptiste.meilhan@ujf-grenoble.fr}
\author[E. Seida]{Eri Seida} 
\address{Tokyo Gakugei University\\
         Department of Mathematics\\
         Koganeishi \\
         Tokyo 184-8501, Japan}
\author[A. Yasuhara]{Akira Yasuhara} 
\address{Tokyo Gakugei University\\
         Department of Mathematics\\
         Koganeishi \\
         Tokyo 184-8501, Japan}
	 \email{yasuhara@u-gakugei.ac.jp}

\thanks{
The first author is supported by the French ANR research project ANR-11-JS01-00201.
The third author is partially supported by a Grant-in-Aid for Scientific Research (C) 
($\#$23540074) of the Japan Society for the Promotion of Science.}
%
%
\maketitle

\begin{abstract} 
A $C_k$-move is a local move that involves $(k+1)$ strands of a link.
A $C_k$-move is called a $C_k^d$-move if these $(k+1)$ strands belong to 
mutually distinct components of a link. 
Since a $C_k^d$-move preserves all $k$-component sublinks of a link, 
we consider the converse implication: are two links with common $k$-component sublinks related by 
a sequence of $C_k^d$-moves? 
We show that the answer is yes under certain assumptions, 
and provide explicit counter-examples for more general situations.  
In particular, we consider $(n,k)$-Brunnian links, i.e. $n$-component links whose $k$-component sublinks are all trivial. 
We show that such links can be deformed into a trivial link by $C_k^d$-moves, thus generalizing a result of Habiro and Miyazawa-Yasuhara, 
and deduce some results on finite type invariants of $(n,k)$-Brunnian links.  
\end{abstract} 
\section{Introduction}
Habiro \cite{H} and Goussarov \cite{G} introduced independently the notion of $C_k$-move, 
which is a local move that involves $k+1$ strands of a link as illustrated in Figure~\ref{cnm}.  
A $C_1$-move is just a crossing change.
Alternatively, a $C_k$-move can be defined in terms of ``insertion'' of elements of the 
$k$th term of the lower central series of the pure braid group \cite{stanford}.

\begin{figure}[!h]
\includegraphics[trim=0mm 0mm 0mm 0mm, width=.8\linewidth]
{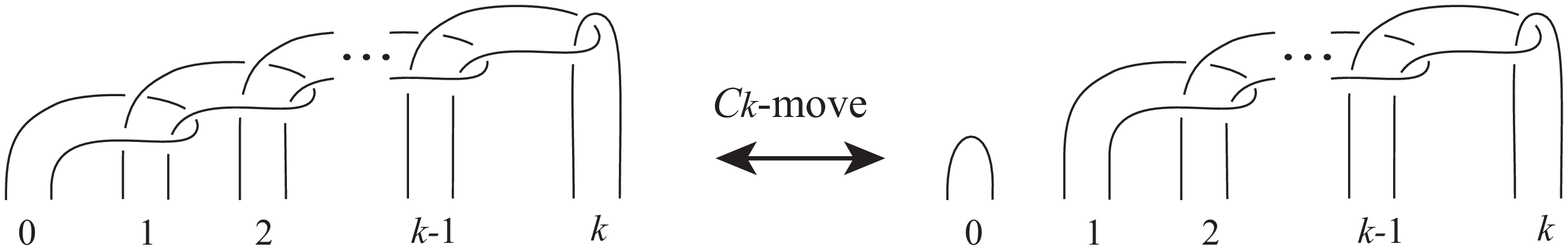}
\caption{A $C_k$-move involves $k+1$ strands of a link. 
} \label{cnm}
\end{figure}

In particular, 
if all $k+1$ strands involved in a $C_k$-move belong to pairwise distinct components, 
we call it a \emph{$C_k^d$-move}.  
The $C_k$-move (resp. $C_k^d$-move) generates an equivalence relation on links, 
called \emph{$C_k$-equivalence} (resp. \emph{$C_k^d$-equivalence}), which becomes finer as $k$ increases.  
It is easy to see that if two links are $C_k^d$-equivalent, then 
they have \emph{common $k$-component sublinks}. 
More precisely, if two ordered links $L=K_1\cup\cdots\cup K_n$ and $L'=K'_1\cup\cdots\cup K'_n$ are 
$C_k^d$-equivalent, then for any subset $S\subset \{1,...,n\}$ with $k$ elements, $\bigcup_{i\in S}K_i$ and $\bigcup_{i\in S}K'_i$ are 
ambient isotopic. 
It seems natural to ask whether the converse implication holds as well.

\begin{question} 
If two links have common $k$-component sublinks, then are they $C_k^d$-equivalent?
\end{question}

Since any link is $C_1^d$-equivalent to a completely split link, the answer is obviously yes 
for $k=1$. Hence we may assume that $k\geq 2$. 

The question can also be given a positive answer for a special class of links. 
An {\em $(n,k)$-Brunnian link} is an $n$-component link whose $k$-component sublinks are trivial \cite{P}. In particular, if $n=k+1$, then 
it is a Brunnian link in the usual sense. 
\begin{theorem}\label{brunnian}
A link is $(n,k)$-Brunnian if and only if it is $C_k^d$-equivalent to the $n$-component trivial link.
\end{theorem}
Theorem~\ref{brunnian} is thus a generalization of the fact that an $(n+1)$-component link is Brunnian if and only if it is $C_{n}^d$-equivalent 
to a trivial link \cite{H2,Miyazawa-Y}. 
This fact is the key ingredient in further works on finite type invariants of Brunnian links \cite{H2,HM,HM2}. 
These results generalize in a straightforward way to $(n,k)$-Brunnian links, see Appendix \ref{ftink}.

For $k=2$, we have a more general statement as follows.
\begin{theorem}\label{d-delta}
Two links with trivial components have common $2$-component sublinks if and only if they are $C_2^d$-equivalent.
\end{theorem}

Although the hypotheses in Theorem \ref{d-delta}, that $k=2$ and that each link component is trivial, may seem restrictive, 
they turn out to be both necessary to give a positive answer to our question. 
Indeed, we have the following.

\begin{proposition}\label{counter-ex}
(1)~For $k\geq 3$, there exists a pair of links with trivial components, 
which have common $k$-component sublinks but are not $C_k^d$-equivalent.  \\
(2)~There exists two links with one nontrivial component, 
which have common $2$-component sublinks and are not $C_2^d$-equivalent. 
\end{proposition}

The rest of the paper is organized as follows. 
In Section 2, we review some elements of the theory of claspers.
We prove Theorem~\ref{brunnian}, Theorem~\ref{d-delta} and Proposition~\ref{counter-ex} in Sections 3, 4 and 5 respectively.   
The paper is concluded by several straightforward extensions of known results on Brunnian links to $(n,k)$-Brunnian links, see Appendix A. 

{\section{Claspers} }
 
We now recall several notions from clasper theory for links \cite{H}.  
In this paper, we only need the notion of tree claspers.  
For a general definition, we refer the reader to \cite{H}.  

\medskip
Let $L$ be a link in $S^3$.  An embedded disk $F$ in $S^3$ 
is called a {\em tree clasper} for $L$ if 
it satisfies the three following conditions: \\
(1) $F$ is decomposed into disks and bands, called {\em edges}, each of which 
connects two distinct disks.\\
(2) The disks have either 1 or 3 incident edges, called {\em leaves} or 
{\em nodes} respectively.\\
(3) $L$ intersects $F$ transversely and the intersections are contained 
in the union of the interior of the leaves.  \\
(In \cite{H}, a tree clasper and a leaf are called 
a {\em strict tree clasper} and a {\em disk-leaf} respectively.) 

A tree clasper is \emph{simple} if each leaf intersects $L$ at one point.  
In the following, \emph{we will implicitely assume that all tree claspers are simple.}

The \emph{degree} of a tree clasper is the number of the leaves \emph{minus} $1$.  
A degree $k$ tree clasper is called a {\em $C_k$-tree} (or a {\em $C_k$-clasper}).  

Given a $C_k$-tree $T$ for a link $L$, there is a procedure to construct  a 
framed link $\gamma(T)$ in a regular neighborhood of $T$. 
\emph{Surgery along $T$} means surgery along $\gamma(T)$.  
Since there exists an orientation-preserving homeomorphism, fixing the boundary, 
from the regular neighborhood $N(T)$ of $T$ to the manifold $N(T)_{T}$ 
obtained from $N(T)$ by surgery along $T$, 
surgery along the $C_k$-tree $T$ can be regarded as a local move on $L$.  
We say that the resulting link $L_T$ is {\em obtained from $L$ by surgery along $T$}. 
In particular, surgery along a $C_k$-tree illustrated in Figure~\ref{milnor-tangle} 
is equivalent to band-summing a copy of the $(k+1)$-component 
Milnor link (see \cite[Fig.~7]{Milnor}), 
and is equivalent to a $C_k$-move as defined in the introduction (Figure~\ref{cnm}).  
Similarly, for a disjoint union $T_1\cup \cdots \cup T_m$ of tree claspers for $L$,  we can define 
$L_{T_1\cup\cdots\cup T_m}$ as the link obtained by surgery along $T_1\cup\cdots\cup T_m$.

\begin{figure}[!h]
\includegraphics[trim=0mm 0mm 0mm 0mm, width=.7\linewidth]
{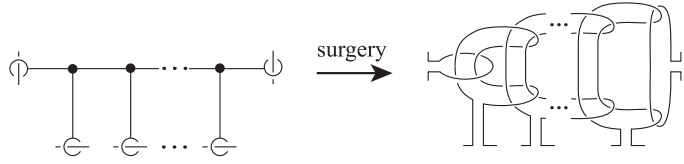}
\caption{Surgery along a $C_k$-tree.} \label{milnor-tangle}
\end{figure}

It is known that the $C_k$-equivalence as defined in Section 1 coincides with 
the equivalence relation on links generated by surgery along $C_k$-trees and ambient isotopy
\cite[Thm.~3.17]{H}. 

Let $L=K_1\cup\cdots\cup K_n$ be an $n$-component link.  
For a $C_k$-tree $T$ for $L$, the set $\{i~|~T\cap K_i\neq\emptyset\}$ is called the 
{\em index} of $T$, and is denoted by $\mathrm{index}(T)$. 

A $C_k$-tree $T$ for $L$ is a {\em $C_k^d$-tree} if it satisfies that 
$|\mathrm{index}(T)|=k+1$, that is, if $|\mathrm{index}(T)|$ is the number of leaves of $T$. 

By arguments similar to those in the proof of \cite[Thm.~3.17]{H}, we have 
that the $C_k^d$-equivalence defined in Section 1 coincides with the equivalence relation on links 
generated by surgery along $C^d_k$-trees and ambient isotopy.  

Although the following three lemmas follow from results of \cite{akira}, the main ideas of proofs are due to Habiro \cite{H}. 

\begin{lemma}\cite[Rem. 2.3]{akira} \label{sliding}
Let $T_1$ and $T_2$ be disjoint $C_1^d$-trees for a link $L$. 
Suppose $\mathrm{index}(T_1)=\{i,j\}$, $\mathrm{index}(T_2)=\{i,k\}$ and $j\neq k$. 
Let $T_1'$ be obtained from $T_1$ by sliding the leaf intersecting the $i$th component of $L$ over the 
leaf of $T_2$ intersecting the $i$th component, see Figure~\ref{figsliding}.
Then $L_{T_1\cup T_2}$ and $L_{T'_1\cup T_2}$ are $C_2^d$-equivalent. \\
\begin{figure}[!h]
\includegraphics[trim=0mm 0mm 0mm 0mm, width=.35\linewidth]
{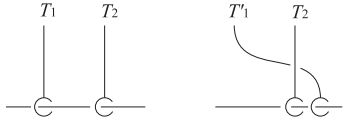}
  \caption{Sliding a leaf over another leaf.}\label{figsliding}
 \end{figure}
\end{lemma}

\begin{lemma}[{\cite[Prop.~2.10]{akira}}, cf. {\cite[Prop.~1.3]{FY}}]
\label{index} 
If an $n$-component link $L'$ is obtained from $L$ by surgery along a $C_k^d$-tree $T$, 
then for any subset $S$ of $\mathrm{index}(T)$ with $|S|\geq 2$, 
there exists a disjoint union of $C_{|S|-1}$-trees for $L$ with index $S$, 
such that $L'$ is obtained from $L$ by surgery.  
\end{lemma}

Combining the latter with \cite[Rem.~2.3]{akira}, we have the following.  
\begin{lemma}\label{cc} 
Let $T$ be a $C_k$-tree for a link $L$, and let $T'$ (resp. $T''$) 
be obtained from $T$ by changing a crossing of an edge of $T$ and the $i$th component of $L$ (resp. an edge of a $C_1$-tree $G$ 
intersecting the $i$th component). 
Then $L_T$ is $C_{k+1}$-equivalent to $L_{T'}$ 
(resp. $L_{T\cup G}$ is $C_{k+1}$-equivalent to $L_{T''\cup G}$), and the $C_{k+1}$-equivalence 
is realized by surgery along  $C_{k+1}$-trees with index $\mathrm{index}(T)\cup\{i\}$. 
\end{lemma}

\section{Proof of Theorem~\ref{brunnian}}

As noted in the introduction, the `if' part of the statement is obvious, so we only need to prove the 'only if' part. 

Let $L=K_1\cup\cdots\cup K_n$ be an $n$-component link 
with same $k$-component sublinks as the trivial $n$-component link $O=O_1\cup\cdots\cup O_n$.
Set $l=n-k$. 
We will show that $L$ is $C_k^d$-equivalent to $O$ by induction on $l$. 

If $l=0$, then $L$ is trivial and the result is obviously true. 

Suppose that $l>0$. 
Let $m$ be the maximum integer so that $L$ is $C_m^d$-equivalent to $O$ 
, i.e., so that there is a disjoint union $F_m$ of $C_m^d$-trees for $O$ such that 
the link $O_{F_m}$ is ambient isotopic to $L$.  
If $m=k$, then we have the result. Hence we assume that $m<k$ and show that this leads to a contradiction. 

Observe that $O_{F_m}$ can be deformed into the split union of $O_1$ and $L\setminus O_1$ by 
deleting all $C_m^d$-trees in $F_m$ intersecting $O_1$ and performing several 
crossing changes between $O_1$ and edges of $C_m^d$-trees disjoint from $O_1$.
By Lemma \ref{cc}, these crossing changes are realized by surgery along $C_{m+1}^d$-trees intersecting $O_1$. 
Since $L\setminus O_1$ is an $(n-1,k)$-Brunnian link, 
by induction hypothesis, it is $C_k^d$-equivalent to a trivial link.  
It follows by Lemma~\ref{index} that 
$L$ is obtained from $O$ by surgery along a disjoint union $F_m^1$ of 
$C_m^d$-trees intersecting $O_1$ and $C_k^d$-trees. 

Similarly, $O_{F^1_m}$ can be deformed into the split union of $O_2$ and $L\setminus O_2$ by 
deleting all tree claspers in $F^1_m$ intersecting $O_2$ and performing several 
crossing changes between $O_2$ and edges of tree claspers in $F^1_m$ disjoint from $O_2$. 
Since $L\setminus O_2$ is an $(n-1,k)$-Brunnian link, 
by induction hypothesis, it is $C_k^d$-equivalent to a trivial link. 
It follows by Lemmas~\ref{cc} and \ref{index} that 
$L$ is obtained from $O$ by surgery along a disjoint union $F_m^2$ of 
$C_m^d$-trees, each intersecting both $O_1$ and $O_2$, and $C_k^d$-trees. 

Repeating this argument inductively, 
we show that $L$ is obtained from $O$ by surgery along a disjoint union $F_m^{m+1}$ of 
$C_m^d$-trees with index $\{1,2,...,m+1\}$, and $C_k^d$-trees. 
Since $m\leq k-1\leq n-2$, $L$ has a $(m+2)$th component, which is 
disjoint from the $C_m^d$-trees in $F_m^{m+1}$. 
We can deform $O_{F^{m+1}_m}$ into the split union of $O_{m+2}$ and $L\setminus O_{m+2}$ by 
deleting all $C_k^d$-trees in $F^{m+1}_m$ intersecting $O_{m+2}$ and several 
crossing changes between $O_{m+2}$ and edges of trees clasper in $F_m^{m+1}$.
Since $L\setminus O_{m+2}$ is an $(n-1,k)$-Brunnian link, 
it is $C_k^d$-equivalent to a trivial link. 
Lemmas~\ref{cc} and \ref{index} then imply that there is disjoint union of $C_{m+1}^d$-trees $F'$ for $O$ such 
that $O_{F'}$ is ambient isotopic to $L$. 
This contradicts the definition of $m$, and thus proves that $L$ is $C_k^d$-equivalent to $O$.

\section{Proof of Theorem~\ref{d-delta}}

The `if' part of the statement is obvious, so we only need to prove the 'only if' part. 

Let $L$ be an $n$-component link with trivial components. 
By \cite{LJ}, there is a diagram of $L$ in ${\Bbb R}^2\times \{0\}$ such that each component has no self crossing.  
For our purpose, it is convenient to further assume that this diagram contains crossings between the $i$th and $j$th components 
for each pair of distinct integers $i, j$ in $\{1,...,n\}$ 
(this is of course possible by a sequence of Reidemeister moves). 

By a sequence of crossing changes, we can deform $L$ into the trivial link $O=O_1\cup\cdots\cup O_n$ 
such that $O_i$ lies in ${\Bbb R}^2\times\{i\}$ $(i=1,...,n)$ and the projections of 
$L$ and $O$ coincide. 
Hence, for each pair of distinct integers $i<j$ in $\{1,...,n\}$, there is a disjoint union $F_{ij}$ 
of $C_1^d$-trees with index $\{i,j\}$ such that $L$ is ambient isotopic to $O_{{\bigcup_{i<j}} F_{ij}}$
(since each crossing in $L$ where the $j$th component underpasses the $i$th one is achieved by surgery on $O$ along such a $C_1^d$-tree).  
Let $p$ be a crossing in the diagram of $O_i\cup O_j$, and let $\alpha_{ij}$ be the arc $p\times[i,j]$. 
(Note that such a crossing always exists by our assumption on the diagram of $L$). 
Since the edge of each $C_1^d$-tree in $F_{ij}$ is contained in ${\Bbb R}^2\times (i,j)$, we have that 
$O_i\cup O_j\cup F_{ij}$ is ambient isotopic to $O_i\cup O_j\cup E_{ij}$, where  
$E_{ij}$ is a disjoint union of $C_1^d$-trees with index $\{i,j\}$ and 
contained in a regular neighborhood of $\alpha_{ij}$. 
It follows, by Lemmas \ref{cc} and \ref{sliding}, that $O_{{\bigcup_{i<j}} F_{ij}}$ is $C_2^d$-equivalent to 
$O_{\bigcup_{i<j}E_{ij}}$.  

Now, let $L'$ be an $n$-component link with trivial components and with same $2$-component sublinks as $L$. 
By the same argument as above, there exists disjoint unions $E'_{ij}$ of $C_1^d$-trees for $O$ with index $\{i,j\}$ 
contained in a regular neighborhood of an arc $\alpha'_{ij}$ in ${\Bbb R}^2\times [i,j]$ $(1\leq i<j\leq n)$ such that 
$L'$ is $C_2^d$-equivalent to $O_{\bigcup_{i<j}E'_{ij}}$.  
By Lemmas \ref{cc} and \ref{sliding}, we may actually assume that the arcs $\alpha_{ij}$ and $\alpha'_{ij}$ coincide. 

Suppose that there is a pair of integer $i,j$ such that $E_{ij}\neq E'_{ij}$. 
We may assume that $(i,j)=(1,2)$. 
Since $L$ and $L'$ have same 2-component sublinks, so do $O_{\bigcup_{i<j}E_{ij}}$ and $O_{\bigcup_{i<j}E'_{ij}}$. 
In particular, $(O_1\cup O_2)_{E_{12}}$ is ambient isotopic to $(O_1\cup O_2)_{E'_{12}}$.  
Notice that this ambient isotopy can be performed in a regular neighborhood of $D_1\cup D_2\cup\alpha_{12}$, where 
$D_1\cup D_2$ is a disjoint union of disks bounded by $O_1\cup O_2$. 
It follows that 
$O\cup \left(\bigcup_{i<j}E'_{ij}\right)$ 
can be deformed into 
$O\cup E_{12}\cup \left( \bigcup_{(i,j)\neq(1,2);i<j}E'_{ij}\right)$ 
by a sequence of moves of the following three types: \\
(i)~leaf sliding (see Figure 2.2) involving two $C_1^d$-trees with index $\{1,2\}$ and $\{1,k\}$ (resp. $\{2,k\}$) for $k\ge 3$, \\
(ii)~ passing the edge of a $C_1^d$-tree with index $\{1,2\}$ across the edge of a $C_1^d$-tree with 
index $\{1,k\}$ (resp. $\{2,k\}$) for $k\ge 3$, \\
(iii)~passing $O_1\cup O_2$ accross all $C_1^d$-trees with index $\{1,k\}$ or $\{2,k\}$ ($k\ge 3$) as illustrated in Figure~\ref{deform1}.\\
Moves of type (i) and (ii) can be achieved by $C^d_2$-moves, according Lemmas \ref{sliding} and \ref{cc} respectively. 
We also have the following: 
\begin{claim}
A move of type (iii) can always be achieved by a sequence of $C^d_2$-moves.  
\begin{figure}[!h]
\includegraphics[trim=0mm 0mm 0mm 0mm, width=.6\linewidth]
{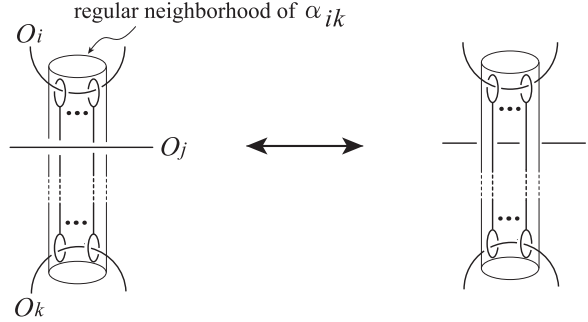}
\caption{Here, $i,j\in\{1,2\}$ are possibly equal, and $k\geq 3$.}
\label{deform1}
\end{figure}
\end{claim}

\begin{proof}
The result follows immediately from Lemma~\ref{cc} when $i\neq j$.  
If $i=j$, Figure~\ref{deform2} shows how the desired deformation can be achieved by passing $O_i$ 
across edges of $\bigcup_{l\neq i}E'_{kl}$, where possibly $k>l$. 
Lemma~\ref{cc} gives us the claim.
\begin{figure}[!h]
\includegraphics[trim=0mm 0mm 0mm 0mm, width=.99\linewidth]
{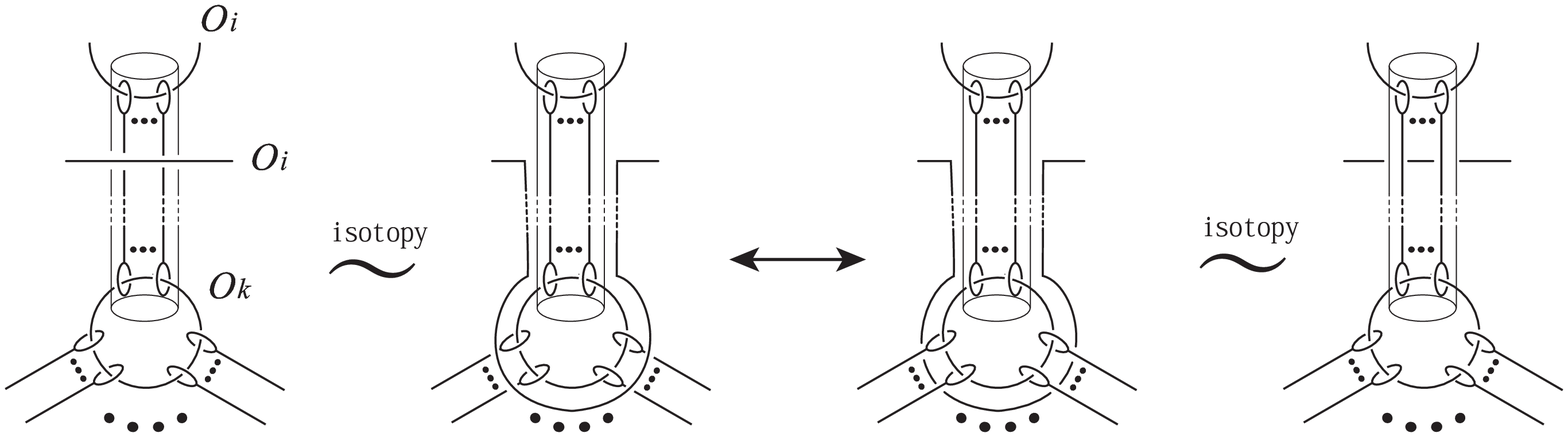}
\caption{}
\label{deform2}
\end{figure}\end{proof}

We conclude that 
$O_{\bigcup_{i<j}E'_{ij}}$ is $C_2^d$-equivalent to $O_{E_{12}\cup \left( \bigcup_{(i,j)\neq(1,2);i<j}E'_{ij}\right)}$.  

Repeating this argument for all pairs $(i,j)$ completes the proof.  

\section{Proof of Proposition \ref{counter-ex}}

In this section, we show that each of the hypotheses imposed in Theorem \ref{d-delta} is necessary for the conclusion to hold. 

\subsection{The case $k\ge 3$: proof of Proposition \ref{counter-ex}~(1)}

We first observe that Theorem \ref{d-delta} does not hold for $k\geq 3$. 

Let $L_n$ and $L'_n$ be two $n$-component links as illustrated in Figure~\ref{example1}. 
Clearly, both links have trivial components, and have same $k$-component sublinks for $k\leq n-1$.
\begin{figure}[!h]
\includegraphics[trim=0mm 0mm 0mm 0mm, width=.6\linewidth]
{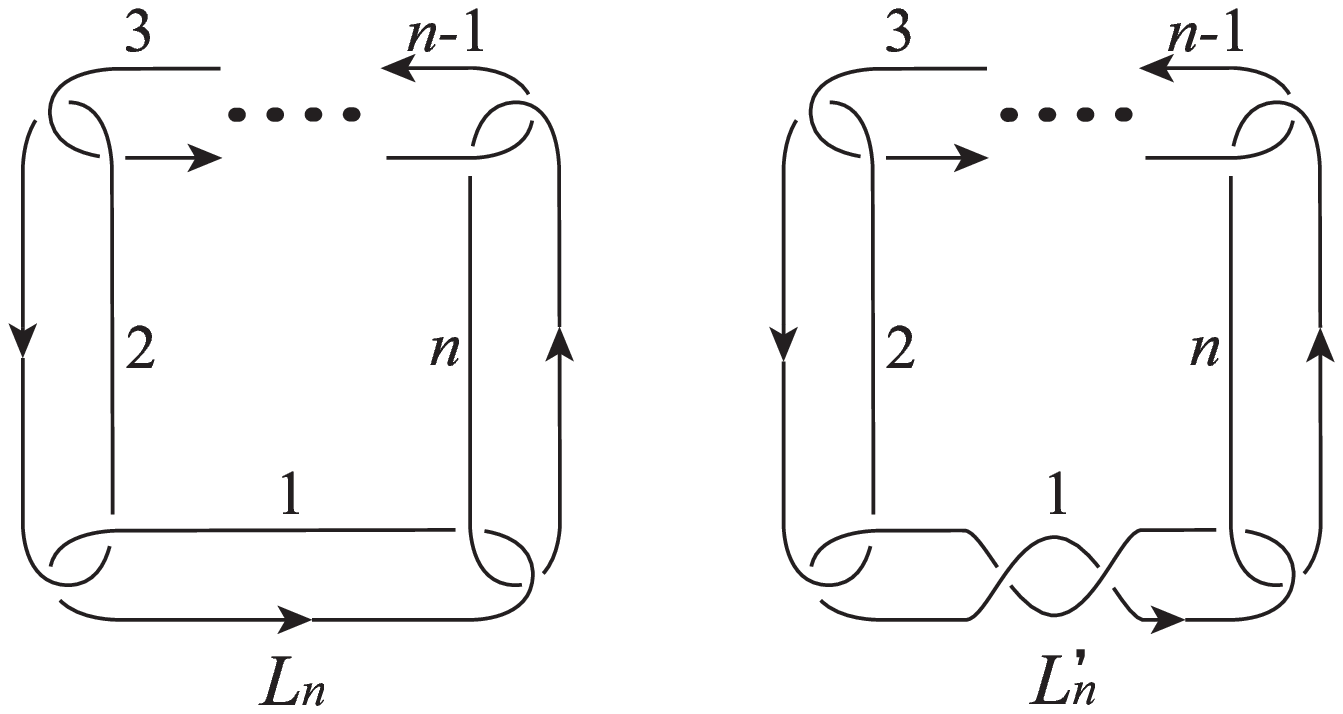}
\caption{ } 
\label{example1}
\end{figure}

On the other hand, we notice that Arf$(L_n)=0$ and Arf$(L'_n)=1$, 
where Arf denotes the Arf invariant \cite{rob}. 
Note that a $C_k$-move preserves the Arf invariant when $k\geq 3$, 
since it can be achieved by a pass-move, which preserves the Arf invariant \cite{MN}.  
This implies that $L_n$ and $L'_n$ are not $C_k$-equivalent, and hence not $C_k^d$-equivalent.

\subsection{An invariant of $C_k^d$-equivalence: proof of Proposition \ref{counter-ex}~(2)}

We now consider the case $k=2$, but without the assumption that all components are trivial. 
For that purpose, we first introduce an invariant of $C_k^d$-equivalence derived from the linking number 
in the double branched cover of $S^3$ branched over a knot.  
%

Let $K\cup K_1\cup\cdots\cup K_m$ $(m\geq 1)$ be an oriented 
$(m+1)$-component link in $S^3$. 
If the linking number $\mathrm{lk}(K,K_i)$ is even 
for all $i(=1,...,m)$, then there is a 
possibly nonorientable surface $F$ bounded 
by $K$ disjoint from $K_1\cup\cdots\cup K_m$. 
Let $G_{\alpha}$ be the {\em Goeritz matrix} \cite{Goe} 
with respect to a basis
$\alpha=(a_1,...,a_n)$ of $H_1(F)$, i.e., 
the $(i,j)$-entry of $G_{\alpha}$ is 
equal to $\mathrm{lk}(a_i,\tau a_j)$, where $\tau a_j$ is a 1-cycle 
in $S^3-F$ obtained by pushing off $2a_j$ in both normal 
directions.  
Let $V_{\alpha}(K_i)=(\mathrm{lk}(K_i,a_1),...,\mathrm{lk}(K_i,a_n))$. 
In \cite{PY} J.~Przytycki and the last author define, for $i,j$ $(1\leq i,j\leq m$), 
\[\lambda_F(K_i,K_j)=
V_{\alpha}(K_i)G_{\alpha}^{-1}V_{\alpha}(K_j)^T,\]
and $\lambda_F(K_i,K_j)=0$ when $F$ is a $2$-disk. 
It follows directly from \cite[Thm~2.3]{PY} that for the double branched cover $M$ of $S^3$ branched over $K$ and 
for lifts $\widetilde{K_i}$ and $\widetilde{K_j}$ of $K_i$ and $K_j$ respectively, we have 
\[\mathrm{lk}_M(\widetilde{K_i},\widetilde{K_j})\equiv \pm \lambda_F(K_i,K_j)~\mathrm{mod}~1.\]
If two links $L=K\cup K_1\cup\cdots\cup K_m$ and $L'=K'\cup K'_1\cup\cdots\cup K'_m$
are $C_k^d$-equivalent for some $k~(2\leq k \leq m)$, 
then $K_i\cup K_j$ and $K'_i\cup K'_j$ are homotopic in the complement of $K$. 
This implies that there is a lift $\widetilde{K_i}\cup \widetilde{K_j}$ (resp. 
$\widetilde{K'_i}\cup \widetilde{K'_j}$) of $K_i\cup K_j$ (resp. $K'_i\cup K'_j$) 
such that 
\[\mathrm{lk}_M(\widetilde{K_i},\widetilde{K_j})\equiv\mathrm{lk}_M(\widetilde{K'_i},\widetilde{K'_j})~\mathrm{mod}~1.\]
It follows that we have the following proposition.

\begin{proposition}\label{inv}
For any $k\geq 2$, $\pm\lambda_F(K_i,K_j)~(\mathrm{mod}~1)$ is an invariant of $C_k^d$-equivalence.
\end{proposition}

We can now complete the proof of Proposition \ref{counter-ex}.

\begin{proof}[Proof of Proposition \ref{counter-ex}~(2)]
Let $L=K\cup K_1\cup K_2$ and $L'=K\cup K_1\cup K'_2$ be links as illustrated in Figure~\ref{example2}. 
Note that $L$ and $L'$ have common 2-component sublinks. 
Let $F$ be a nonorientable surface, and let $a_1,a_2$ be a basis of $H_1(F)$ as illustrated in Figure~\ref{example2}. 
Then we have $\lambda_F(K_1,K_2)=0$ and $\lambda_F(K_1,K'_2)=-{1}/{3}$. 
Proposition~\ref{inv} implies that $L$ and $L'$ are not $C_2^d$-equivalent.
\begin{figure}[!h]
\includegraphics[trim=0mm 0mm 0mm 0mm, width=.5\linewidth]
{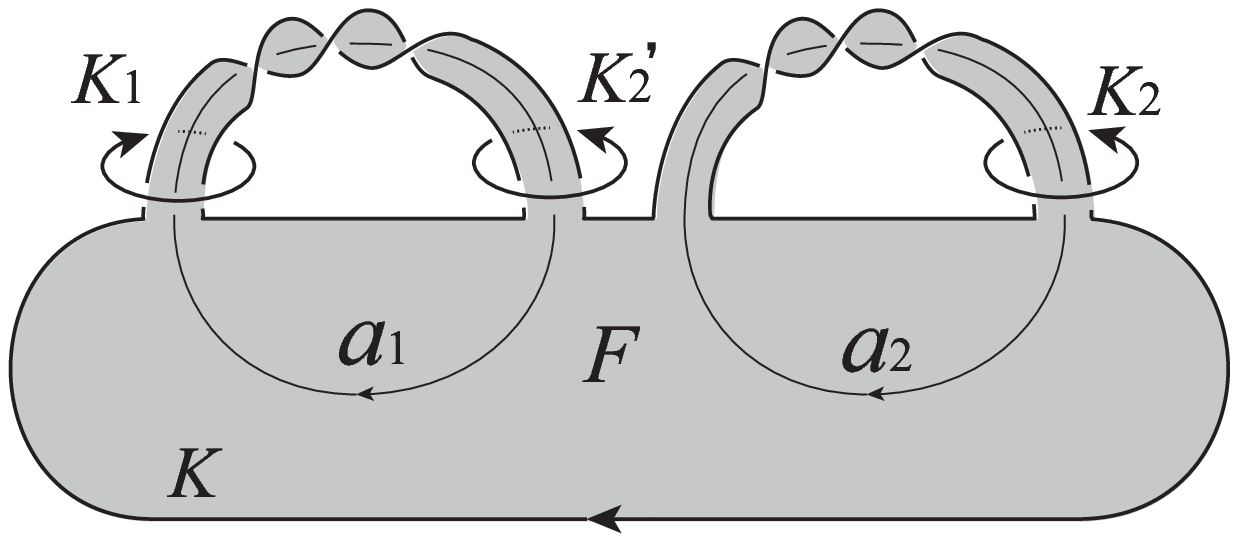}
\caption{}
\label{example2}
\end{figure}\end{proof}  

\appendix

\section{Finite type invariants of $(n,k)$-Brunnian links}\label{ftink}

Theorem \ref{brunnian} states that a link is $(n,k)$-Brunnian if and only if it is $C^d_k$-equivalent to the $n$-component trivial link. 
As recalled in the introduction, the case $k=n-1$, i.e. the case of Brunnian links, was shown in \cite{H2,Miyazawa-Y},  
and is a key ingredient in proving several results and Brunnian links and their finite type invariants.  
Using Theorem \ref{brunnian}, we can easily generalize these to $(n,k)$-Brunnian links.  
We only provide statements here, since the proofs are straightforward generalizations of \cite{H2,Miyazawa-Y,HM,HM2}, 
and require no new idea.  

In \cite{H2}, Habiro shows that for $n\geq 3$, an 
$n$-component Brunnian links cannot be distinguished from the trivial link by any finite type invariant of order less than $2(n-1)$. 
(Note that for $n=2$, this does not hold since the Hopf link and the 2-component trivial link can be 
distinguished by the linking number, which is of order 1.) 
By the same arguments as those in \cite[\S~4]{H2}, we have that if a link is $C_k^d$-equivalent 
to a trivial link for $k\geq 2$, then these links cannot be distinguished by any finite type invariant of 
order less than $2k$. Hence we obtain the following result.
\begin{theorem}
For $n>k\geq 2$, $(n,k)$-Brunnian links and the $n$-component trivial link cannot 
be distinguished by any finite type invariant of order less than $2k$.
\end{theorem}

In \cite{HM,HM2}, the study of finite type invariants of Brunnian links is continued, by expressing the restriction 
of an invariant of degree $2n-1$ to $n$-component Brunnian links as a quadratic form on the Milnor link-homotopy invariants of length $n$, 
see \cite{Milnor}.  
The arguments used in \cite{HM} (and \cite{HM2}) can be generalized in a straightforward way to $(n,k)$-brunnian links to prove the following. 
\begin{theorem}
  Let $f$ be any finite type link invariant of degree $2k+1$ taking values in an abelian group $A$.  
  Then there are (non-unique) elements 
  $f^{\sigma ,\sigma '}_{I}\in A$ 
  for $\sigma ,\sigma '$ in the symmetric group $S_{k-1}$ on the set
  $\{1,\ldots ,k-1\}$ and for any subsequence 
  $I$ of $12...n$ of length $k+1$, such that, for any $(n,k)$-Brunnian link
  $L$, the difference $f(L)-f(O)$ is equal to 
  \begin{equation*}
   \sum_{\substack{I=i_1i_2...i_{k+1} \\ \textrm{subseq. of } 12...n}} 
   \sum_{\sigma,\sigma'\in S_{k-1}} 
   f^{\sigma ,\sigma '}_{I}\bar\mu _L(i_{\sigma (1)}\ldots i_{\sigma (k-1)} i_k i_{k+1})\bar\mu _L(i_{\sigma' (1)} \ldots i_{\sigma' (k-1)} i_k i_{k+1}). 
  \end{equation*}
Here $O$ denotes the $n$-component trivial link and $\bar \mu_L$ denotes Milnor invariants of~$L$.  
\end{theorem} 
\noindent In other words, the restriction of an invariant of degree $2k+1$ to $(n,k)$-Brunnian links can be expressed as 
a quadratic form on the Milnor link-homotopy invariants of its $(k+1)$-components sublinks, and, in particular, 
is determined by the $(k+1)$-component sublinks.   

Milnor invariants are useful not only for understanding finite type invariants of $(n,k)$-Brunnian links, but also for providing classification results.  
It is indeed known that $n$-component Brunnian links are classified up to $C_n$-equivalence by Milnor link-homotopy invariants 
\cite{Miyazawa-Y,HM}.  
Again, strictly similar arguments can be used to extend this classification result as follows.  
\begin{theorem}
 Two $(n,k)$-Brunnian links are $C_{k+1}$-equivalent if and only if they cannot be distinguished by any Milnor 
link-homotopy invariant of length $k+1$.  
\end{theorem}
\end{document}